\documentclass[11pt]{amsart}
\usepackage{times}
\usepackage[T1]{fontenc}
\usepackage{amssymb, amsthm, amsmath}
\usepackage{mathrsfs}
\usepackage{hyperref}
\usepackage[active]{srcltx}

\usepackage{todonotes}

\usepackage{setspace}
\usepackage{geometry}

\usepackage[T1]{fontenc}
\usepackage[latin9]{inputenc}
\usepackage{color}
\usepackage{textcomp}
\usepackage{mathrsfs}
\usepackage{stmaryrd}
\usepackage{stackrel}
\usepackage{graphicx}
\usepackage{subscript}

\makeatletter

\newcommand{\CM}{\mathcal M}
\newcommand{\CN}{\mathcal N}
\newcommand{\CL}{\mathcal L}
\newcommand{\sub}{\subseteq}
\newcommand{\0}{\emptyset}

\newcommand{\aff}[1]{{\mathscr A}{(#1)}}
\theoremstyle{plain}
\newtheorem{thm}{\protect\theoremname}[section]

\newtheorem*{theorem*}{Theorem}
\theoremstyle{definition}
\newtheorem{defn}[thm]{\protect\definitionname}
\theoremstyle{plain}
\newtheorem{fact}[thm]{Fact}
\theoremstyle{remark}
\newtheorem{claim}[thm]{\protect\claimname}

\theoremstyle{plain}
\newtheorem{lem}[thm]{\protect\lemmaname}
\theoremstyle{plain}
\newtheorem{prop}[thm]{\protect\propositionname}
\theoremstyle{remark}
\newtheorem{rem}[thm]{\protect\remarkname}

\makeatother

\providecommand{\claimname}{Claim}
\providecommand{\definitionname}{Definition}

\providecommand{\lemmaname}{Lemma}
\providecommand{\propositionname}{Proposition}
\providecommand{\remarkname}{Remark}
\providecommand{\theoremname}{Theorem}
\newcommand{\CR}{\mathcal R}
\newcommand{\CI}{\mathcal I}
\newcommand{\ra}{\rangle}
\newcommand{\la}{\langle}

\newcommand{\rest}{\upharpoonright}

\title{Additive Reducts of real closed fields and strongly bounded structures}
\author{Hind Abu Saleh}
\address{Department of Mathematics, University of Haifa, Haifa, Israel}
\email{\href{mailto:hind.abu.94@gmail.com}{hind.abu.94@gmail.com}}

\author{Ya'acov Peterzil}
\address{Department of Mathematics, University of Haifa, Haifa, Israel}
 \email{\href{mailto:kobi@math.haifa.ac.il}{kobi@math.haifa.ac.il}}

\begin{document}

\begin{abstract}
Given a real closed field $R$, we identify exactly four proper reducts of $R$ which expand the underlying (unordered) $R$-vector space structure. Towards this theorem we introduce a new notion, of strongly bounded reducts of linearly ordered structures:

A reduct  $\CM$ of
a linearly ordered structure $\la R;<,\cdots\ra $ is called
\emph{strongly bounded} if every $\mathcal M$-definable subset of
$R$ is either bounded or co-bounded in $R$.
We investigate strongly bounded additive reducts of o-minimal structures and as a corollary prove the above theorem on additive reducts of real closed fields.
\end{abstract}

\thanks{The second author is partially supported by the Israel Science Foundation grant 290/19.}
\maketitle

\section{Introduction}

The motivation behind the work here is a conjecture about reducts of real closed fields, from
\cite{Pet1}. Before stating the conjecture, let us clarify our usage
of the notion of ``reduct'' here.

\begin{defn}
Given two structures $\mathcal{M}$ and $\mathcal{N}$, we say that
$\mathcal{M}$ is \emph{reduct of} $\mathcal{N}$ (or, $\mathcal{N}$
is \emph{an expansion of} $\mathcal{M}$), denoted by $\CM\dot{\subseteq} \CN$, if $\mathcal{M}$ and $\mathcal{N}$
have the same universe and every set that is definable in $\mathcal{M}$
is also definable in $\mathcal{N}$ (where definability allows parameters).
We say that $\mathcal{M}$ and $\mathcal{N}$ are \emph{interdefinable}, denoted by $\mathcal{M}\dot{=}\mathcal{N}$,
if $\mathcal{M}$ is reduct of $\mathcal{N}$ and $\mathcal{N}$ is
reduct of $\mathcal{M}$.

We say $\mathcal{M}$ is \emph{a proper reduct} of $\mathcal{N}$
(or, $\mathcal{N}$ a proper expansion of $\mathcal{M}$) if $\mathcal{M}\dot{\subseteq}\mathcal{N}$
and not $\mathcal{M}\dot{=}\mathcal{N}$.
\end{defn}

Below, we let $\Lambda_R$ be the family of all $R$-linear maps $\lambda_a(x)=\alpha x$, for all $\alpha\in R$. Our ultimate goal here is to prove:
\begin{thm}\label{theorem: main} Let $R$ be a real closed field.
Then, the only reducts between the vector space $\langle R;+,\Lambda_R \rangle $ and the field $\langle R;<,+,\cdot \rangle $ are
\begin{center}
$\mathcal R_{alg}:=\langle R;<,+,\cdot\rangle $
\end{center}
\vspace{.1cm}

\begin{center}
$\mathcal{R}_{sb}:=\langle R;<,+,\Lambda_R, \mathfrak B_{sa} \rangle $
\par\end{center}
\vspace{.1cm}

\begin{center}
$\mathcal{R}_{semi}:=\langle R;<,+,\Lambda_R \rangle $
\hspace{.4cm} $\CR_{bd}:=\langle R;<^{*},+,\Lambda_R,\mathfrak B_{sa}\rangle $
\end{center}
\vspace{.1cm}
\begin{center}
$\CR_{lin}^*:=\langle R;<^{*},+,\Lambda_R\rangle $
\par\end{center}

\vspace{.1cm}
\begin{center}
$\CR_{lin}:=\langle R;+,\Lambda_R \rangle $,
\par\end{center}
where $<^{*}$ is the linear order on the interval $(0,1)$ and $\mathfrak B_{sa}$ the collection of all bounded semialgebraic sets over $R$.
\end{thm}

\begin{rem}
\begin{enumerate}
\item The definable sets in $\CR_{alg}$ are called semialgebraic, while those definable in $\CR_{semi}$ are semilinear. The structure $\CR_{sb}$ above is called {\em semibounded}, as it expands the ordered vector space by a collection of bounded sets. Semibounded structures were studied in several articles, for example \cite{Edmundo},\cite{Belegradek},\cite{Pet2}.

\item Notice that because all the above structures expand the full underlying $R$-vector space,
then once $<^{*}$ is definable then the restriction of $<$ to every bounded interval is definable.

\item A similar project, in the setting of Presburger Arithmetic, was carried out in \cite{Conant}, where it was proven that there are no proper reducts between $\la \mathbb Z;+\ra$ and $\la \mathbb Z;<,+\ra$. We expect that in arbitrary models of Presburger arithmetic, an analoguous result to Theorem \ref{theorem: main} holds, with the intermediate reducts corresponding to possible restrictions of $<$ to infinite subintervals.
\end{enumerate}
\end{rem}

Some of the work towards the proof of Theorem \ref{theorem: main} can be
read off earlier results. In particular, the fact that the semibounded reduct $\CR_{sb}$ is the only proper reduct between $\CR_{semi}$ and $\mathcal R_{alg}$, was proven over $\mathbb R$  in \cite{Pet1} and can be deduced for arbitrary real closed field from Edmundo's \cite{Edmundo} (see Fact \ref{fact:edmundo} below).  However, the bulk of the work here is to show that if a reduct $\CM$ of $\CR_{alg}$ does not define the full order then it is necessarily a reduct of $\CR_{bd}$. Towards that, we introduce
a new notion, of ``a strongly bounded structure'' in a more general setting, and most of our results here
are about such structures:
\begin{defn}
Let $\CR=\langle R;<,\cdots \rangle $ be
a linearly ordered structure. A reduct  $\CM=\langle R;\cdots \rangle $
of $\mathcal R$ is  called
\emph{strongly bounded} if every $\CM$-definable
$X\sub R$ is either bounded or co-bounded (namely, $R\setminus X$ is bounded).
\end{defn}
\begin{rem}\begin{enumerate} \item  The term ``strongly bounded'' was chosen to reflect a combination of a semibounded structure with a strongly minimal one.
Almost all of our work here concerns strongly bounded additive reducts of o-minimal structures, where the underlying linear order is dense. Analogous definitions could be given for, say, models of Presburger arithmetic if one wishes to study all reducts which expand the underlying ordered group.
\item The definition of a strongly bounded structure requires an ambient linear order, thus it might not seem amenable to working in elementarily equivalent structures. However, in practice we only work in sufficiently saturated elementary extensions of a strongly bounded $\CM$ as above, and thus we may assume that this elementary extension is also a reduct of a linearly ordered elementary extension of $\CR$.
    \end{enumerate}
    \end{rem}
\vspace{.2cm}

By definition, if $\CM$ is a strongly bounded reduct of a linearly ordered structure then the ordering $<$ is not definable in $\CM$. We prove several results about strongly bounded reducts of o-minimal structures (see for example Theorem \ref{thm: main results} and Theorem \ref{acl=dcl}):
\begin{theorem*} Let $\la R;<,+,\cdots\ra $ be an o-minimal expansion of an ordered group
and let $\CM=\la R,+,\cdots\ra$ be a strongly bounded reduct. Then
\begin{enumerate}
\item Every $\CM$-definable subset of $R^n$ is already definable in the structure $\la R;+,\Lambda_\CM,\mathfrak B^*\ra$, where $\Lambda_M$ is the collection of $\CM$-definable endomorphisms of $\la R,+\ra$ and $\mathfrak B^*$ is the collection of all $\CM$-definable bounded sets.

    \item For every $\CN\equiv \CM$, the model theoretic algebraic closure equals the definable closure.
        \end{enumerate}
        \end{theorem*}

\noindent{\bf Acknowledgments}  The article are based on results from the  M.Sc. thesis of the first author, as part of her studies at the University of Haifa.
We thank Itay Kaplan and Assaf Hasson for reading an earlier version of this work and commenting on it.

\section{Proper expansions of ${\mathcal R}_{lin}$}

In this section we assume that $\CR_{omin}$ is an o-minimal expansion of a real closed field $R$ and $\CM=\la R;+,\cdots\ra$ is an additive reduct of $\CR_{omin}$.




\begin{thm}
\label{thm:<^* is definable in the vector space} If $\CM$ is not a redcut of $\CR_{lin}=\la R;+,\Lambda_R\ra$ then $<^{*}$
is definable in $\mathcal{M}$.
\end{thm}
\begin{proof} It is sufficient to prove that some interval $[0,b]$ is $\mathcal{M}$-definable, for
$b>0$.

\begin{claim}
\label{claim:-is-unstable}
$Th(\mathcal{M})$ is unstable
\end{claim}
\begin{proof} This is based on the work of  Hasson and Onshuus with the second author, \cite{HOP}.

Assume towards contradiction that $Th(\mathcal{M})$ is stable.
By \cite[Theorem 1]{HOP}, every 1-dimensional stable structure interpretable
in an o-minimal structure is necessarily 1-based. So $\mathcal{M}$
is 1-based. By \cite[Theorem 4.1]{HP}, it follows that every $\CM$-definable set
is a boolean combination of cosets of definable subgroups of $R^{n}$. Every
definable subgroup of $\la R^n;+\ra$ in an
o-minimal structure is an $R$- vector subspace of $R^{n}$ and therefore every $\CM$-definable set is definable in $\CR_{lin}$,
a contradiction. Hence $\mathcal{M}$ is unstable.
\end{proof}

Because $\CM$ is unstable, it is in particular
not strongly minimal. This generally implies that in some
elementary extension of $\CM$, we have an $\CM$-definable subset in one variable which is infinite
and co-infinite. However, o-minimal structures eliminate $\exists^{\infty}$, and therefore so does $\CM$. It follows that there is some $\CM$-definable subset of $R$ itself which is infinite and co-infinite. Call this set $Y$.

By o-minimality, $Y$ has the following form:
\begin{equation}
\label{eq:1} Y:=I_{1}\cup I_{2}\cup \cdots\cup I_{n}\cup L,
\end{equation}
such that for every $i\in\{1,...,n\}, I_{i}:=(a_{i},b_{i})$,
$L\:$ is a finite set and in addition $-\infty \leq a_{1}<b_{1}<a_{2}<...<a_{n}<b_{n}\leq+\infty$.
Without loss of generality $L=\0$.

The following will be used in several places in this thesis.
\begin{lem}
\label{Lemma full linear order is definabe-1} Assume that $Y\sub R$ is definable in an o-minimal expansion of an ordered group. If both $Y$ and $R\setminus Y$
are unbounded then full linear order is definable in $\la R;+,Y\ra$.
\end{lem}

\begin{proof}
If both $Y$ and $R\setminus Y$ are unbounded
then $Y$ has the form (\ref{eq:1}) above and then without loss of generality,
we may assume that $I_{1}=(-\infty,b_{1})$, and $I_{i}=(a_{i},b_{i})$
for $i\in\{2,...,n\}$.

By replacing $Y$ by $Y-b_{1}$ we may assume that $b_{1}=0$ and
then
\[
-Y\cap Y=(-b_{n},-a_{n})\cup...\cup(-b_{2},-a_{2})\cup(a_{2},b_{2})\cup...\cup(a_{n},b_{n})
\]
 So $(-Y\cap Y)\cap[(-Y\cap Y)+(a_{n}+b_{n})]$ equals the interval
$I_{n}=(a_{n},b_{n})$ in $Y$. Replace $Y$ by $Y_{1}:=Y\setminus I_{n}$, now $Y_{1}$ contains a unbounded ray together with $n-2$ bounded
intervals. Continuing in this way we obtain a ray $(-\infty,0)$ that
is definable, so we can define $<$.\end{proof}

So we assume now that either $Y$ or $R\setminus Y$ are bounded. If $Y$ is bounded then each of the
intervals above  in $Y$ is bounded. Let $\alpha:=b_{n}-b_{1}$. The set $\:(Y+\alpha)\cap Y$
defines a single interval whose right endpoint is $b_{n}$. So, we
are done.
If $Y$ is unbounded then replace
$Y$ by $R \setminus Y$ and finish as before.
Hence, we showed that $<^{*}$
is definable in $\CM$.\end{proof}

\section{Reducts of $\CR_{alg}$ which are not semilinear}

Here $R$ is a real closed field and $\CR_{alg}=\la R;<,+,\cdot\ra$.
Before the next theorem we recall previous work from \cite{LP} (see a corrected and more general proof in \cite{Belegradek}),
which will be used in its proof.

Given $a>0$ in $ R$, let $I=(-a,a)$. Denote by $+^{*}$
the partial function obtained by intersecting the graph of $+$ with
$I^{3}$, and for each $\alpha\in R$, let $\lambda_{\alpha}^{*}$
be the partial function obtained by intersecting the graph of $\lambda_{\alpha}$
with $I^{2}$. Finally, let $<^{*}$ be the restriction of $<$ to
$I^{2}$. Notice that for each $X\subseteq R^{n}$ such that $\la R;<+,\cdot,X\ra$ is o-minimal,
the structure $$\mathcal{I}=\langle \mathrm{I};<^*,+^{*},\{\lambda^*_{\alpha}\}_{\alpha\in R}, X\cap I^n \rangle $$
is o-minimal as well.


In \cite{LP} the structure $\langle I;<^*,+^{*}\rangle $
was called a group-interval and its o-minimal expansions were studied
there.

A partial endomorphism (p.e. in short) of this group-interval was
a function $f:I\rightarrow I$ which respects addition when defined:
namely, if $x,y,x+^{*}y\in I$ then $f(x+^{*}y)=f(x)+^{*}f(y)$.

Notice that in our setting every $\mathcal{I}$-definable
p.e. is necessarily the restriction of $\lambda_{\alpha}$ for some
$\alpha\in R$. Indeed, if $f:I\rightarrow I$ is
an $\mathcal{I}$-definable p.e. then it is not hard to verify that
the following is a semialgebraic subgroup of $\langle R,+ \rangle $
which contains all integers. Let
\[
H=\{r\in R:\exists\varepsilon>0\forall x\in(-\varepsilon,\varepsilon)f(rx)=rf(x)\}.
\]

O-minimality of the real field implies that $H= R$ and therefore $f$ is the
restriction of an $R$-linear map, namely the restriction
of $\lambda_{\alpha}$ for some $R$.

Now, without going through the precise definition of ``a linear theory''
from \cite{LP}, it was shown in \cite[Proposition 4.2]{LP}
that if $Th(\mathcal{I})$ is linear then every $\mathcal{I}$-definable
set is already defined in the structure $\langle I;+^{*},<^{*},\{\lambda_{\alpha}^*\}_{\alpha\in R}\rangle ,$ (together
possibly with additional parameters). Thus if $Th(\mathcal{I})$ is
linear then $X\bigcap I^{n}$ is a semilinear set.

The following proposition seems to be obvious but for the
sake of completion we include a proof in the Appendix.

\begin{fact}
\label{prop: there exists open box}
Let $R$ be a real closed field and  $X\subseteq R^{n}$
a definable set in an o-minimal expansion of $\la R;<,+,\cdot \ra$.
If $X$ is not semilinear then, in the structure $\CM=\la R;<^*,+,\Lambda_R,X\ra$,
there exists a definable bounded set which is not semilinear.
\end{fact}

\begin{thm}
\label{thm4:If--is}If $X\subseteq R^{n}$
is semialgebraic and not definable in $\mathcal{R}_{semi}$, then
every bounded $R$-semialgebraic set is definable in $\langle R;+,\Lambda_R,X \rangle$
\end{thm}
\begin{proof}
Let $\mathcal{M}:= \langle R;+,\Lambda_R,X \rangle $.
By Theorem \ref{thm:<^* is definable in the vector space}, the relation
$<^{*}$ is definable in $\mathcal{M}$. Let us first see that $\CM$
defines a real closed field on some interval.

By Fact \ref{prop: there exists open box}, we may assume that $X\cap I^n$ is not semilinear, for some bounded interval $I=(-,a,a)$.
Consider the o-minimal structure $$\mathcal{I}:=\langle I;<^{*},+^*,\{\lambda_{\alpha}^{*}\}_{\alpha\in R},X\cap I^{n} \rangle,$$
as we described before stating the theorem. We noted
that if $Th(\mathcal{I})$ is linear then the set $X\bigcap I^{n}$
must be semilinear set. Because $I^{n}\cap X$ is not semilinear
then $Th(\mathcal{I})$ is not linear in the sense of \cite{LP}
and therefore by  \cite[Theorem 1.2]{PS1}, a real closed field is $\CI$-definable, hence also $\CM$-definable, on some interval $J\subseteq I$.

Without loss of generality, assume that $J=(-a_{0},a_{0}),\:a_{0}>0$.
Denote the field by:
\[
\mathcal{J}=\langle \mathit{J},\varoplus,\varodot \rangle
\]

The structure $\mathcal{J}$ is $\mathcal{M}$-definable.
By \cite[Corollary 2.4]{Pet1}, every $R$-semialgebraic subset
of $J^{k}$, $k\in\mathbb{N}$, is definable in $\mathcal{J}$,
and therefore in $\mathcal{M}$.

Let  $B\subseteq(-b,b)^{n}$ for some $b>0$ in $ R$. Using
scalar multiplication from $\Lambda_R$,  we can contract $(-b,b)$  into $(-a_{0},a_{0})$,
so it is definable in $\mathcal{J}$. It follows that $B$
is definable in $\mathcal{M}$ .
\end{proof}

\section{Strongly bounded structures }

The ultimate goal of this section is to prove:
\begin{thm}
\label{thm:< is definable in R}  Let $R$ be a real closed field. If $X\sub R^n$
is semialgebraic and not definable in $\CR_{bd}=\langle R;<^{*},+,\Lambda_R,\mathfrak B_{sa}\rangle $
then $<$ is definable in the structure $\langle R;+,\Lambda_R, X \rangle $.
\end{thm}


We are going to work in a more general setting than that of a real closed field.
Recall that a strongly bounded reduct of a linearly ordered $\la R;<,\cdots \ra$  is one in  which every definable subset of $R$ is bounded or co-bounded.
Below, we will mostly be interested in strongly bounded reducts of o-minimal structures.
By Lemma \ref{Lemma full linear order is definabe-1} we
have:
\begin{lem}\label{Lemma:stronglu bounded} Let $\CR_{omin}=\la R;<,+, \cdots\ra$ be an o-minimal expansion of an ordered group.
If $\mathcal{M}=\langle R;+,\cdots \rangle $
is a reduct of $\CR_{omin}$ then
$\mathcal{M}$ is strongly bounded if and only if $<$ is not definable
in $\mathcal{M}$.
\end{lem}

So in order to prove Theorem \ref{thm:< is definable in R} it is sufficient to prove that if $X\sub R^n$ is definable in a strongly bounded $\mathcal{M}=\langle R;<,+,\cdots \rangle$ then $X$ is definable in $\la R;+,\Lambda_\CM, \mathfrak B_\CM\ra $, where $\mathfrak B_\CM$ is the collection of  all $\CM$-definable bounded sets. A more precise and slightly stronger theorem will be proved soon, Theorem \ref{thm: main results}. We first make a general observation which we shall exploit repeatedly.


\subsection{Definability of ``boundedness''}
For $X\subseteq T\times R^{n}$, $T\subseteq R^{m}$ and  $t\in T$, we let
\[
X_{t}=\{a\in R^n :\langle t,a\ra \in X\}
\]



The following general result will be very useful here.
\begin{prop}
\label{prop:bounded-definable}
Let $\mathcal{M}=\langle R;+,\cdots \rangle $ be any reduct of an o-minimal expansion of an ordered group. If $\{X_{t}:\:t\in T\}$ is
an $\mathcal{M}$-definable family of subsets of $R^{n}$,
then the set $$\{t\in T:\:X_{t}\:\textrm{is\:bounded in $R^n$}\}$$ is definable
in $\mathcal{M}$.
\end{prop}
\begin{proof}
Note that a set $Y\subseteq R^{n}$
is bounded if and only if for each $i$, the image of $Y$ under the
projection map $\pi_{i}:\left\langle y_{1},...,y_{n}\right\rangle \longmapsto y_{i}$
is bounded in $R$. Thus, it is sufficient to prove the result under the assumption that all $X_t$ are subsets of $R$.

By o-minimality, each $X_t\sub R$ is unbounded if and only it contains an unbounded ray. Thus,
it is easy to see that
\[
\{t\in T:\:X_{t}\textrm{ is bounded}\}=\{t\in T:\:\exists a\:\:a+X_{t}\cap X_{t}=\emptyset\}
\]
and hence the set is definable in $\mathcal{M}$.
\end{proof}

\subsection{The strongly bounded setting}
We first clarify and somewhat generalize our setting.
\label{setting assumptions}

\vspace{.2cm}

\emph{Let $\CR_{omin}=\la R,<,+,\cdots \ra$ denote an o-minimal expansion of an ordered group in language $\CL_{omin}$, and let $\CM=\la R;+,\cdots\ra$ denote a strongly bounded reduct of $\CR_{omin}$, in language $\CL$, such that $acl_\CM(\0)$ contains at least one nonzero element.}

\begin{defn} An interval $(a,b)\sub R$ is called a \emph{$\0$-interval in $\CM$} if $a,b\in acl_\CM(\0)$. A subset $X\sub R^n$ is called {\em $\0$-bounded
in $\CM$} if $X$ is contained in some $I^n$, for $I$ a $\0$-interval in $\CM$.
\end{defn}

  {\em Our standing assumption is that for every $\0$-interval $I\sub R$, the restricted order $<\rest \! \! I$ is $\0$-definable in $\CM$.
  Notice that, using Theorem \ref{thm:<^* is definable in the vector space},  this is  true when $\CM$ is elementarily equivalent to a reduct of a real closed field which properly expands $\CR_{lin}$.}

  \vspace{.2cm}

  We let $\Lambda_{\CM}$ be the collection of all $\CM$-definable
endomorphisms of $\la R,+\ra$, defined over $\0$.
We let $\CL_{bd}(\CM)$ be the language consisting of $\{+,\{\lambda\}_{\lambda\in \Lambda_{\CM}}\}$,  augmented by a predicate for every $\0$-definable, $\0$-bounded set in $\CM$.

By expanding $\CL$ and $\CL_{omin}$ by function symbols and predicates for $\0$-definable sets, we may assume that
$$\CL_{bd}\sub \CL\sub \CL_{omin}.$$ We let $\CM_{bd}$ be the reduct of $\CM$ to $\CL_{bd}$.

\vspace{.2cm}
Our ultimate goal in this section is to prove:
\begin{thm}\label{thm: main results} For $\CM$ strongly bounded as above, every definable subset of $R^n$ is definable in $\CM_{bd}$.\end{thm}

One of our main difficulties in working with strongly bounded structures is the failure of global cell decomposition.
E.g. the set $R\setminus \{0\}$ cannot be decomposed definably into definable cells in a strongly bounded structure, because no ray is definable there.

Another difficulty is the fact that a-priori we do not know whether the model theoretic algebraic closure equals the definable closure in strongly bounded structures.
However, we shall eventually show, see Theorem \ref{acl=dcl}, that $acl=dcl$ in this setting.

We assume from now on throughout this section that $\CM$ is strongly bounded as above.
\subsection{Definable subsets of $R$ in strongly bounded structures}


Notice that although the full order is not definable in $\CM$, a basis for the $<$-topology on $R$ and the product topology on $R^n$ is
 definable in $\CM$, using the restricted order. Thus we have:
\begin{lem}
\label{lem clusure} If $\{X_t:t\in T\}$ is an $\CM$-definable family of subsets of $R^n$, then the families
$$\{Cl(X_t):t\in T\}\,\,,\,\, \{Int(X_t):t\in T\}\,\,,\,\,\{Fr(X_t):t\in T\} $$ are definable in $\CM$.

\end{lem}

Every $\CM$-definable $X\sub R$ is a union of finitely many pairwise disjoint \textbf{maximal} open sub-intervals of $X$ (which are possibly not $\CM$-definable)
 and a finite set.  Below, when we say that {\em $I$ is an interval in $X$} we mean that $I$ is one of these open components of $X$.

\begin{defn}
Let $Y\subseteq R$ be an $\mathcal{M}$-definable set, we
define:
\vspace{.1cm}

$\partial^{-}(Y):=\{y\in R:\:y\:\textrm{is a left endpoint of an interval in }Y\}.$

\vspace{.1cm}

$\partial^{+}(Y):=\{y\in R:\:y\:\textrm{is a right endpoint of an interval in }Y\}.$
\end{defn}

\begin{lem}
\label{lem:endpoints}If
$\{Y_{t}:\:t\in T\}$ is an $\mathcal{M}$-definable family
of bounded subsets of $R$ then the families
$\{\partial^{-}(Y_{t}):\:t\in T\},\:\{\partial^{+}(Y_{t}):\:t\in T\}$
are $\mathscr{\mathcal{M}}$-definable, over the same parameter set.
\end{lem}
\begin{proof} We fix an $\CM$-definable $<\rest\!\! (0,a_0)$ for some $a_0>0$.
We define$\:\partial^{-}(Y_{t})$ by the formula :
\[
(x\notin Y_{t}\wedge\exists\epsilon<a_{0}\:(x,x+\epsilon)\subseteq Y_{t})
\]
\[
\vee
\]
$$(x\in Y_{t}\wedge\exists\epsilon\leq a_{0}\:\wedge(x-\epsilon,x)\cap Y_{t}=\emptyset\wedge(x,x+\epsilon)\subseteq Y_{t})).$$

Because of the definability of $<^{*}$ in $\mathscr{\mathcal{M}}$,
$\{\partial^{-}(Y_{t}):t\in T\}$ is $\mathcal{M}$-definable. We
similarly handle $\partial^{+}(Y_{t})$.
\end{proof}
The next theorem is an important component of our analysis of strongly
bounded structures.
\begin{thm}
\label{thm uniform bound on the length} If $\{X_{t}:\:t\in T\}$
is an $\mathcal{M}$-definable family of bounded subsets
of $R$, then there is a uniform bound on the length of each
interval in $X_{t}$. Moreover, there exists such a bound in $dcl_\CM(\0)$.
\end{thm}

\begin{proof} By Proposition \ref{prop:bounded-definable}, every $\CM$-definable family $\{X_t:t\in T\}$ of bounded subsets of $R$ is a sub-family of a $\0$-definable
family of such sets. Namely, if $\phi(x,t,a)$ is the formula defining the $X_t$'s over $a$, as $t$ varies, then we can consider the formula
$$\psi(x,t,y):\phi(x,t,y)\wedge \mbox{ $\psi(R,t,y)$ is a bounded set}.$$ Thus, it is sufficient to prove the result for $\0$-definable families.

By Lemma \ref{lem clusure}, we may
assume that each $X_{t}$ is an open set.
We will use induction on $n:=\textrm{the maximum number of intervals in }X_{t}$,
for $t\in T$.

For $n=1$, write  $X_{t}=(a_{t},b_{t})$.

 Consider the family
$\{X_{t}-a_{t}:t\in T\}$.
By Lemma \ref{lem:endpoints},
the family is $\0$-definable.


 Thus, the  set $Y=\underset{t\in T}{\bigcup}X_{t}-a_{t}$ is an $\mathcal{M}$-definable interval, over $\0$, whose left end-point is $0$.
 Because $\CM$ is strongly bounded, this interval must be bounded, hence its right endpoint is some $K\in M$.
 By Lemma \ref{lem:endpoints}, the point $K$ is definable over $\0$.


Consider now the case $n=k+1$, namely each $Y_t$ consists of at most $k+1$ pairwise disjoint open intervals.
For each $t\in T$, let $D_{t}=\{c_{_{1}}-c_{2}:\:c_{1},c_{2}\in\partial^{-}(X_{t})\}$,
an $\mathcal{M}$-definable set by Lemma \ref{lem:endpoints}.\\

\begin{claim}
\label{claim:For-each-,}For each $t\in T$, there exists $d\in D_{t}$
such that $(X_{t}+d)\cap X_{t}$ is one of the intervals in $X_{t}$.
\end{claim}
\begin{proof}
Let $X_{t}=I_{1,t}\cup I_{2,t}\cup...\cup I_{k+1,t},\textrm{ where each }I_{m,t}:=(a_{m,t},b_{m,t})$,
such that:
\[
a_{1,t}<b_{1,t}<a_{2,t}<b_{2,t}<...<a_{k+1,t}<b_{k+1,t}
\]
For an interval $I=(a,b)$, let $|I|=b-a$.

Let $d=a_{k+1,t}-a_{1,t}$. In the set $X_{t}+d$, for each $m$,
the interval$\textrm{ }I_{m,t}$ is shifted to$\textrm{ }I_{m,t}+d$.
So $(X_{t}+d)\cap X_{t}$ consists of either $I_{k+1,t}$ (when $|I_{k+1,t}|<|I_{1,t}|$)
or $I_{1,t}+d$ (when $|I_{k+1,t}|>|I_{1,t}|$) .

If it consists of $I_{k+1}$ we are done. Otherwise we take
\[
d'=a_{1,t}-a_{k+1,t}\in D_{t}
\]
 and then $(X_{t}+d')\cap X_{t}=$ $I_{1,t}$.

So in both cases there exists $d\in D_{t}$ such that $X_{t}+d\cap X_{t}$
is one of the intervals in $X_{t}$.
\end{proof}

We define the set:
\[
D'_{t}:=\{d\in D_{t}:\:(X_{t}+d)\cap X_{t}\:\textrm{is one of the intervals in }X_{t}\}
\]

\begin{claim}
The family $\{D'_{t}\::t\in T\}$ is an $\mathscr{\mathcal{M}}$-definable
family of nonempty sets.
\end{claim}
\begin{proof}
For $t\in T$, $d\in D'_{t}$ if and only if the following two hold:

\begin{enumerate}
\item $\partial^{-}((X_{t}+d)\cap X_{t})\subseteq\partial^{-}(X_{t})\:\textrm{and }|\partial^{-}((X_{t}+d)\cap X_{t})|=1$, and

\item $\partial^{+}((X_{t}+d)\cap X_{t})\subseteq\partial^{+}(X_{t})\:\textrm{and }|\partial^{+}((X_{t}+d)\cap X_{t})|=1$.
\end{enumerate}

By Lemma \ref{lem:endpoints},
$(1),(2)$ are a definable properties in $\mathscr{\mathcal{M}}$.
By Claim \ref{claim:For-each-,}, each $D_{t}'$ is non-empty.
\end{proof}

We proceed with the proof of Theorem \ref{thm uniform bound on the length}.
Consider the $\mathcal{M}$-definable family
\[
\{\:Y_{t,d}:=X_{t}+d\cap X_{t}:\:d\in D'_{t},\:t\in T\}
,\] still defined in $\CM$ over $\0$.
 For every $t\textrm{ and }d\in D'_{t},$ the set $Y_{t,d}$ consists
of a single interval which is one of the intervals in $X_{t}.$ By
case $n=1$ we know that there is a uniform bound on length of each
$Y_{t,d}$, call it $w_{1}$ which can be chosen in $dcl_{\CM}(\0)$.
We now define, still over $\0$, the following family
\[
\{Z_{t,d}:=X_{t}\setminus\:Y_{t,d}:\:d\in D'_{t},\:t\in T\}
\]
 Each subset $Z_{t,d}$ consists of at most $k$ intervals among the
$k+1$ intervals of $X_{t}$. By the induction hypothesis, we know
that there is a uniform bound on the length of each interval, call
it $w_{2}$ which we may choose in $dcl_{\CM}(\0)$. 

Thus the maximum of $w_1,w_2$, which is in $dcl_\CM(\0)$,  is the bound on the
length of each interval of $X_t$, as $t$ varies. This
ends the proof of Theorem \ref{thm uniform bound on the length}.\end{proof}

As a corollary we can now match, definably in $\mathcal{M}$,
each left endpoint of an interval in $X_{t}$ with the corresponding
right endpoint:
\begin{prop}
\label{Prop:match endpoints} Let
$\{X_{t}:\:t\in T\}$ be an $\mathcal{M}$-definable family
of bounded subsets of $R$, and let
\[
L_{t}=\{\langle a,b\rangle \in\partial^{-}(X_{t})\times\partial^{+}(X_{t})\textrm{: the interval }(a,b)\textrm{ is one the intervals of }X_{t}\}
\]
 Then the family $\{L_{t}:t\in T\}$ is $\mathcal{M}$-definable.
\end{prop}
\begin{proof} By Theorem \ref{thm uniform bound on the length}, there is a bound, call it $K\in dcl_\CM(\0)$, for the length of each interval in $X_{t}$, for
all $t\in T$.
 For each $t\in T\:$, we have
\[
(1) \langle a,b \rangle \in L_{t}\:\:\iff\:\:a\in\partial^{-}(X_{t})\textrm{ and \ensuremath{b=min(\partial^{+}(X_{t})\cap[a,a+K])} }
\]
By Lemma \ref{lem:endpoints},
$\partial^{-}(X_{t}),\:\partial^{+}(X_{t})$ are definable families
and since in $(1)\:$ we only use the order on $[0,K]$, the family
$\{L_{t}:t\in T\}$ is definable in $\mathscr{\mathcal{M}}$.
\qed

\begin{rem}\begin{enumerate}

\item Notice that Theorem \ref{thm uniform bound on the length} fails without the assumption that the $X_t$'s are bounded sets. Namely, it is
not true in general that the length of the bounded components of $X_t$ is bounded in $t$.
 For example, the sets $X_t=R\setminus \{-t,t\}$ has $(-t,t)$ as an open component, with unbounded length as $t\to \infty$.

 Also, even if each $X_t$'s is bounded it is not true that the diameter of the $X_t$'s is uniformly bounded.
For example, take the family $\{(-t,t-1)\cup (t,t+1):t\in R\}$ that is definable using $<\upharpoonright (0,1)$.

 \item We do not know whether Proposition \ref{Prop:match endpoints}
 holds if we drop the assumption that the $X_t$'s are bounded. Namely, can we still match definably the left and right endpoints of the bounded components of $X_t$,
  when the $X_t$'s are unbounded?

  \end{enumerate}
  \end{rem}

\subsection{Affine sets and functions}
Recall that  $\CR_{omin}$ is an o-minimal expansion of an ordered divisible abelian group $R$, and
we assume that $\CM=\la R;+,\cdots\ra $ is a strongly bounded reduct of $\CR_{omin}$ in which $<$ is $\0$-definable on every $\0$-interval.
We let $<^*$ denote the ordering on some fixed interval we call $(0,1)$.


\begin{defn}\label{affine} Let $\la R;<,+\ra$ be an abelian ordered divisible group. \begin{enumerate} \item A map $f:R^n\to R^k$ is {\em affine} if it
 is of the form $\ell(x)+d$ for $\ell:R^n\to R^k$ a homomorphism between $\la R^n,+\ra$ and $\la R^k,+\ra$,  and $d\in R^k$.
  \item A (partial) function $f: R\rightarrow R$ , is \emph{eventually
affine} if there exists $a>0$ such that $(a,\infty)\sub dom(f)$ and  the restriction of $f$ to
$(a,+\infty)$ is affine.

\item $X\sub R^n$ is {\em locally affine at $a\in X$} there is an open neighborhood $U\ni a$ such that for all $x,y,z\in U\cap X$, $x-y+z\in X$. The {\em affine part of $X$} is the set;
\[
\mathscr{A}(X)=\{x\in X:\:X\textrm{ is locally affine at }x\}.
\]
\end{enumerate}
\end{defn}

Notice that if $X$ is the graph of an affine map then $\mathscr{A}(X)=X$.
Also, because a basis for the $R^n$-topology is definable in $\CM$, we immediately have:
\begin{lem} \label{theorem:the affine part of a family }  Let
 $\{X_t:t\in T\}$ be an $\CM$-definable family of subsets of $R^n$, defined over $\0$.
Then the
family $\{\mathscr{A}(X_{t}):\:t\in T\}$ is $\mathcal{M}$-definable, over $\0$.
\end{lem}


We now prove:

\begin{prop} \label{prop: finite number of slopes} Every $\CM$-definable endomorphism $f:R\to R$ is $\0$-definable.\end{prop}
\begin{proof} Assume that  $f$ is defined by   $\CM$-formula $\phi(x,y,a)$, over the parameter $a$.
We will show that $f$ can be defined without parameters.

Since being an $R$-endomorphism is $\CM$-definable, we may assume that there is some $\CM$-definable $T\sub R^k$, such that
for all $t\in T$. If $\phi(R^2,t)$ is non-empty then it defines
a  non-zero endomorphism $f_t$ of $\la R;+\ra$.

 Assume first that the set of endomorphisms $f_t$'s defined by $\phi$ is finite. Define $t_1 E t_2$ iff $f_{t_1}=f_{t_2}$, an $\CM$-definable equivalence relation.
Consider the functions near $0$, and define $[t_1]_E<[t_2]_E$ if for all $x>0$ sufficiently small, we have $f_{t_1}(x)<f_{t_2}(x)$. By o-minimality, we obtain
a linear ordering of the finitely many $E$-classes, and since $<$ is $\CM$-definable in a neighborhood of $0$, this ordering is $\CM$-definable.
Thus, each $f_{t}$ in this finite family of endomorphisms is $\0$-definable.

Assume now that the  family $\{f_t:t\in T\}$ is infinite, and we shall reach a contradiction.
Consider the set $\{f_t(1):t\in T\}$. By o-minimality it contains an open interval $(a,b)$, and by replacing each $f_t$ with $f_t-f_{t_0}$,
for some $t_0\in T$ for which $f_{t_0}\in (a,b)$, we may assume that the interval $(a,b)$ contains $0$ and the ordering on $(a,b)$ is $\CM$-definable (we think of $f_t(a)$ as ``the slope'' of $f_t$). Let $T_0=\{t\in T:f_t(1)\in (0,b)\}$.

We write $t_1 \sim t_2$ if $f_{t_1}=f_{t_2}$, and let $[t]$ be the equivalence class of $t$. In abuse of notation we let $f_{[t]}$ denote the corresponding endomorphism of $R$.

  By o-minimality, if $f_{t_1}(1)=f_{t_1}(1)$ then $f_{t_1}=f_{t_2}$, thus we obtain an $\CM$-definable function $t:(0,b)\to T_0/\sim$, defined by $f_{[t(x)]}(1)=x$. Namely, $f_{[t(x)]}$ is the endomorphism whose ``slope'' is $x$.
   Fix an element $d>0$, and define
 $\sigma:(0,b)\to R$ by: $\sigma(x)=f_{[t(x)]}^{-1}(d)$. Namely, $\sigma(x)=y$ if there exists $t\in T_0$ such that $f_t(1)=x$ and $f_{t}(y)=d$ (we may think of $\sigma (x)$ as ``$d/x$'').
 The function $\sigma$ is also $\CM$-definable. For every $t\in T_0$, we have $f_t(1)>0$, hence
  $f_t(x)>0$ if and only if $x>0$. Threfore, $\sigma$ is positive on $(0,b)$.

\vspace{.1cm}
\noindent {\em Claim} $Im(\sigma)$ is unbounded in $R$.

\vspace{.1cm}

 Indeed, assume towards contradiction that $K=sup(Im(\sigma))<\infty $. By our observation, $K>0$.
 Choose $y_0\in Im(\sigma)$, $y_0<K$ and  sufficient close to $K$,  such that $K<2y_0$. By assumption, there exists $t_0\in T_0$ and $x_0>0$,  such that $f_{t_0}(1)=x_0$ and
 $f_{t_0}(y_0)=d$.

Let $t_1\in T_0$ be such that $[t_1]=t(x_0/2)$. Then $f_{t_1}(1)=x_0/2=f_{t_0}(1)/2$. It follows that $f_{t_1}=f_{t_0}/2$ and hence $$f_{t_1}(2y_0)=f_{t_0}(2y_0)/2=f_{t_0}(y_0)=d$$ But then $f_{t_1}(1)=x_0/2$ and $f_{t_1}(2y_0)=d$, so by definition,  $\sigma(x_0/2)=2y_0>K$,
contradicting the assumption that $K$ bounds $Im(\sigma)$.

Thus, $Im(\sigma)$ is an $\CM$-definable set which is unbounded and positive, contradicting the assumption that $\CM$ is strongly bounded.\end{proof}

 \begin{defn} We denote by $\Lambda_{omin}$ the set of all $\CR_{omin}$-definable endomorphism $f:\la R,+\ra\to \la R,+\ra$.
 and we still let $\Lambda_\CM$ denote the set of all $\CM$-definable endomorphisms of $R$, which by Proposition \ref{prop: finite number of slopes},
 is necessarily $\0$-definable. Let $\Lambda_{omin}^*$ and $\Lambda_\CM^*$ denote those non-zero endomorphisms.

 \end{defn}

\subsection{Definable functions of $1$-variable}

Our goal is to describe definable functions in 1-variable, and prove that $\CM$ has no  definable ``poles''.
\begin{prop}
\label{prop eventually constant} If $g:R\rightarrow R$
is an $\mathcal{M}$-definable partial function whose domain
is co-bounded and $Im(g)$ is bounded. Then $g$ is constant on a co-bounded set.
\end{prop}
\begin{proof} By o-minimality, there exists $L\in R$ such that
$\underset{x\rightarrow+\infty}{lim}g(x)=L$. We shall see that $g\equiv L$
on a co-bounded set.

The function $g$ is definable in an o-minimal structure, thus there exists $a_{1}\in R$
such that $g\rest (a_{1},+\infty)$ is either constant or strictly
monotone, and there exists $a_{2}$ such
that $g$ is constant or strictly monotone on $(-\infty,a_{2})$.

If $g$ is constant $L$ on $(a_{1},+\infty)$ then $\{x\in R:g(x)=L\}$
is unbounded and since $\mathcal{M}$ strongly bounded the set must be co-bounded
and we are done. Assume towards contradiction that $g\rest(a_1,\infty)$ is strictly monotone.

Assume first that $g$ is strictly increasing on $(a_1,\infty)$. Notice that the property of
being locally increasing in a neighborhood of $x\in R$ is definable using $<^*$, thus the set
$$\{x\in R: g \mbox{ is locally increasing at $x$}\}$$ is $\CM$-definable, contains $(a_1,\infty)$ and hence must be co-bounded.
It follows that $g$ is strictly increasing on $(-\infty, a_2)$.

Because $\lim_{x\to \infty} g(x)=L$ and $g$ is increasing,  there exists $b\in R$ such that
for all $x>b$, $L-1<g(x)<L$. Because $<^*$ is $\CM$-definable the set of all $x\in R$ such that $L-1<g(x)<L$
is $\CM$-definable so must be co-bounded. In particular, we may assume that  $L-1<g(x)<L$ for all $x<a_2$
and thus $g(x)$ has a limit $L_1\in R$ as $x\to -\infty$.

But since $g$ is increasing on $x<a_{2}$, it follows that $L_1<L$ and in addition
there exists $a_2'\leq a_2$ and $\epsilon>0$, such that for all $x<a_2'$, $$L_1<g(x)<L_1+\epsilon<L.$$

Using $<^*$ again, this is an $\CM$-definable property of $x$
so must hold also for all $x>a_1'$, contradicting the fact that $\lim_{x\to +\infty} g(x)=L$.

A similar argument works when $g$ is eventually decreasing.\end{proof}

\begin{rem} By \cite{Edmundo}, if $\CN=\la R;<,+,\cdots\ra $ is an o-minimal expansion of an ordered group in which every definable bounded function is eventually constant then $\CN$ is {\em semibounded},  namely every definable set is definable using the underlying vector space, together with all the definable bounded sets. This might suggest  a fast deduction of Theorem \ref{thm: main results} from Proposition \ref{prop eventually constant}. The problem of this approach is that we do not know  that the definable functions in the strongly bounded $\CM=\la R;+,<^*,\cdots \ra$ are the same as in its expansion by the full $<$. Thus, we do not see how to apply Edmundo's theorem here.
\end{rem}


Next, using almost identical arguments to Edmundo's \cite{Edmundo}  we shall show
that every $\CM$-definable function $f:R\to R$ is affine on a co-bounded set. For that, we recall some notation and facts, based on work
of Miller and Starchenko \cite{MS}.

\vspace{.2cm}

\noindent{\bf Notation}
 For $\CR_{omin}$-definable positive (partial) functions
$f,g:R\to R$, such that $(a,\infty)\sub \mbox{dom}(f),\mbox{dom}(g)$, we write $f\leq g$ (or $f<g$) if $f(x)\leq g(x)$ (or, $f(x)<g(x)$) for
all large enough $x$.

We write $v(f)<v(g)$) if $|f|>|\lambda\circ g|$ for
all $\lambda\in \Lambda_{omin}^*$ such that $\lambda>0$. We also write $v(f)=v(g)$ if there are $\lambda_1,\lambda_2\in \Lambda_{omin}^*$, both positive, such that
$$|\lambda_1\circ g|\leq |f|\leq |\lambda_2\circ g|.$$ This easily seen to be an equivalence relation.

Finally, we write $\Delta(f)=f(x+1)-f(x)$.

\begin{fact}\cite{Edmundo}\label{Edmundo} For every $\CR_{omin}$-definable function on unbounded ray.
\begin{enumerate}
\item If $v(f)>v(x)$ then $\lim_{x\to \infty}\Delta(f)=0$.

\item If $v(f)<v(x)$ then $v(f^{-1})>v(x)$.

\item If $v(f)=v(x)$ then $\Delta(f)(x)$ has a limit in $R$ as $x\to \infty$.
\end{enumerate}
\end{fact}

\vspace{.2cm}

The following is just a warm-up towards Theorem \ref{thm: subset of R^2}.
The proof follows closely the proof of  \cite[Poposition 2.8]{Edmundo}, which uses results of Miller and Starchenko \cite{MS}:

\begin{lem}
\label{lemma eventually affine} If $f: R\rightarrow R$
is $\mathcal{M}$-definable on a co-bounded set, then $f$ is eventually affine. Moreover, there exists a $\0$-definable endomorphism
$\lambda\in \Lambda_\CM$ and $A>0$ such that for all $x$ with $|x|>A$, we have $f(x)=\lambda(x)+d$, for some $d\in R$.
\end{lem}
\begin{proof} Assume towards contradiction that $f:R\to R$ is not eventually affine.
Without loss of generality, $f$ is eventually increasing, and by Proposition \ref{prop eventually constant},
it must approach $+\infty$. If $v(f)>v(x)$ then by Fact \ref{Edmundo}, $\lim_{x\to \infty}\Delta(f)=0$. Since $\Delta(f):=f(x+1)-f(x)$ is definable
in $\CM$, it follows from Proposition \ref{prop eventually constant} that it must be eventually $0$ and therefore  $f$ is eventually affine.

If $v(f)<v(x)$ then by \ref{Edmundo}, $v(f^{-1})>v(x)$, where $f^{-1}$ is taken to be the eventual compositional inverse of $f$, which is also definable in $\CM$.
Thus, as above, $f^{-1}$ is eventually affine so also $f$ is.

We are left with the case $v(f)=v(x)$. By Fact \ref{Edmundo} (3), the $\CM$-definable function $\Delta(f)$ approaches a limit $c$ in $R$. By Proposition
\ref{prop eventually constant}, we have $\Delta(f)$ eventually constant, and thus, by o-minimality, $f$ is eventually affine.

Thus, we showed so far that there exists a definable endomorphism $\lambda\in \Lambda_\CM$  such that $f(x)=\lambda(x)+d$ for all $x>0$ large enough.
By Proposition \ref{prop: finite number of slopes}, $\lambda$ is $\0$-definable. The set $$\{x\in R:f(x) =\lambda(x)+d\}$$ is $\CM$-definable
and contains an bounded ray so must be co-bounded.\end{proof}

Before the next proposition, we introduce a new notion.
\begin{defn}
Given $X\subseteq R^{n}$, let
\[
Stab_{bd}(X):=\{a\in R^{n}:\:(a+X)\triangle X\:is\:bounded\},
\]

where $A\triangle B=A\cup B\setminus A\cap B$.
\end{defn}

{\em For a function $f$, we let $\Gamma(f)$ denote its graph}.

By Proposition \ref{prop:bounded-definable}, if $X$ is definable in $\CM$ over $A$ then so is $Stab_{bd}(X)$.
The following are easy to verify:
\begin{fact} \label{fact:(1)ON STAB}
\begin{enumerate}
\item  For every $X\subseteq R^{n}$,
$Stab_{bd}(X)$ is a subgroup of $\la R^{n},+\ra $.

\item If $X\subseteq R ^{2}$ is the graph of an affine function
$f(x)=\lambda(x)+b$, on a co-bounded subset of $R$, then
\[
Stab_{bd}(X)=\Gamma(\lambda).
\]

\item If a definable set $X\subseteq R^{2}$ is a finite union
of graphs of affine functions, all of the form $\lambda+d$ for a fixed $\lambda$, and at least one of the
 functions is defined on an unbounded set then $Stab_{bd}(X)=\Gamma(\lambda)$.\end{enumerate}
\end{fact}

The following statement would have been immediately true if definable sets in $\CM$ admitted definable cell decomposition (with respect to the ambient ordering).
\begin{prop}\label{0-definable} Assume that $X\sub R^2$ is $\CM$-definable over $A$, and $\dim(X)\leq 1$. Assume that there exists an $\CR_{omin}$-definable
endomorphism $\lambda:R\to R$,
and some $a,d\in R$ such that graph of $\lambda(x)+d\rest(a,\infty)$ is contained in $X$. Then $\lambda$ is $\CM$-definable (necessarily over $\0$).\end{prop}
\begin{proof}
 Recall that $\mathscr{A}(X)$, the affine part of $X$ is $\CM$-definable over $A$. For large enough $a$, it contains $\Gamma(\lambda+d\rest \, (a,\infty))$. So, without
loss of generality, $X=\mathscr{A}(X)$.

We define for each $x,y\in X$,
the relation $x\sim y$ iff there exist open sets $U,V\ni 0$ in $R^2$, such that
$$(y-x)+(x+U\cap X)=y+V\cap X.$$
Said differently, up to translation, $X$ has the same germ at $x$ and at $y$.
Because a basis for the $R^2$ topology is definable in $\CM$, the relation $\sim$ is definable in $\CM$.

Notice that for $x$ large enough, all elements on $\Gamma(\lambda+d)\cap X$ are in the same $\sim$-class, so we may replace $X$ by this $\sim$-class, which is $\CM$-definable.




Thus, we may assume that
all elements of $X$ are $\sim$-equivalent, and $X$ contains $\Gamma(\lambda+d\rest(a,\infty))$. It follows that $X$ is contained in finitely many translates of the graph of $\lambda$. Applying Fact \ref{fact:(1)ON STAB}(3), we conclude that $Stab_{bd}(X)$ is exactly the graph of $\lambda$,
thus the function $\lambda(x)$ is $\CM$-definable. By Lemma \ref{prop: finite number of slopes}, $\lambda$ is $\0$-definable.\end{proof}

\subsection{Definable subsets of $R^2$}
The next result is the main structure theorem of the paper.

\begin{thm}\label{thm: subset of R^2}Under our standing assumptions on $\CM$.

Assume
that $X\subseteq R^{2}$ is definable in $\mathcal{M}$ over a parameters set $A\sub R$, with $\dim(X)\leq 1$.

Then, there are $\lambda_1,\ldots,\lambda_r\in \Lambda_\CM$, and there are $\CM$-definable finite set $D_i\sub R$, $i=1,\ldots,r$, and $D\sub R $ all defined over $A$,
such that

(i) for every $i=1,\ldots, r$, and $d\in D_i$,
$\Gamma(\lambda_i+d)\setminus X$ is bounded (i.e. $X$ contains the restriction of $\lambda_i+d$ to a co-bounded set).

(ii) For every $d\in D$, $(\{d\}\times R)\setminus  X$ is bounded.

(ii) The set $$X\setminus (\bigcup_{i=1}^r\bigcup_{d\in D_i} \Gamma(\lambda_i+d)\, \cup\, \bigcup_{d\in D} \{d\}\times R)$$ is bounded in $R^2$.

\end{thm}

\begin{proof}  If $X$ is bounded then there is nothing to prove so we assume $\dim(X)=1$ and $X$ is unbounded.
By the cell decomposition theorem in o-minimal structure, $X$ can be decomposed into
a finite union of cells of dimension 0 and 1. However, these
cells are not in general definable in $\mathcal{M}$.






Assume first that $X$ contains the graph of a function $f:(a,+\infty)\to R$,
and let $\Psi(x,y)$ be the $\CM$-formula that defines $X$.

\vspace{.2cm}
\noindent{\bf Case (i)}  $f$ is bounded at $\infty$.
\vspace{.2cm}

In this case we prove a general statement:

\begin{claim}\label{unbounded}  If $\dim X\leq 1$ and $X$ contains the graph of a bounded function $f:(a,\infty)\to R$ then $f$ is eventually constant.\end{claim}
\begin{proof} By o-minimality, $\lim_{x\to +\infty}f(x)=L$ for
some $L\in R$.

By our standing assumption, $<\rest (0,a_0)$ is $\CM$-definable, for some $a_0>0$, and thus $<$ is definable on every interval of length $\leq a_0$.
 Let $X_{L}:=R\times[L-a_{0},L+a_{0}])\cap X$.
By o-minimality,  there exists $m\in \mathbb N$, such that for all large enough $a\in R$, we have $|X_a|\leq m$.
The set $Z=\{a\in R: |X_a|\leq m\}$ is definable in $\CM$ and unbounded, thus we may replace $X_L$ by $X_L\cap Z\times R$, containing the graph of $f$.  We  call it $X_L$ again.

Using the restricted order, we can partition $X_{L}$, definably in  $\CM$, into
finitely many graphs of functions $g_{1},g_{2},\dots ,g_{k}$, $k\leq m$. E.g, we let
\[
g_{1}(x)=min\{y\in[L-a_0,L+a_{0}]:\:\langle x,y\rangle\in X_{L}\}
\]
and continue similarly to obtain the other $g_i$'s. For $x$ large
enough, the function $f$ is one of those $g_{i}$'s, therefore it is $\CM$-definable. Using Proposition \ref{prop eventually constant}
we get that $f$ is eventually constant.\end{proof}

\vspace{.2cm}
\noindent
\textbf{Case (ii):} $\underset{x\rightarrow+\infty}{lim}f(x)=+\infty$:

\vspace{.2cm}
We recall the proof of Lemma \ref{lemma eventually affine}, and consider three cases:
$v(f)>v(x)$, $v(f)<v(x)$ and $v(f)=v(x)$
(remembering though that we do not know yet that $f$ is an $\CM$-definable function).

Assume first that $v(f)>v(x)$. By Fact \ref{Edmundo},  $f(x+1)-f(x)\to 0$, as $x\to \infty$. We want to capture
$\Delta(f)=f(x+1)-f(x)$ within an $\CM$-definable set.

The formula $$\varphi(x,y):=\exists z_{1}\exists z_{2}(\Psi(x+1,z_{1})\wedge\Psi(x,z_{2})\wedge(y=z_{1}-z_{2})),$$
defines in $\CM$ a new subset of $R^2$ call it $\Delta(X)$ which contains the graph of $\Delta(f)$
(but possibly more functions).

We first note that $\dim(\Delta(X))=1$:
Indeed,  for $a\in R$, $\Delta(X)_{a}$ is infinite if either $X_a$ or $X_{a+1}$ is infinite. Since only finitely many $X_a$'s are infinite
the same is true for $\Delta(X)$.
Thus, the graph of $\Delta(f)$ is contained in the one-dimensional $\CM$-definable set $\Delta(X)$,  so by Claim \ref{unbounded}, $\Delta(f)$
must be eventually constant, implying that $f$ is eventually affine.

Assume now that $v(f)<v(x)$. The formula $\Upsilon(x,y):=\Psi(y,x)$
defines in $\CM$ a new set $X^{-1}$ containing in it the graph of $f^{-1}$ (a partial
function).
 The graph  of  $f^{-1}$ is still contained in $X^{-1}$ and we have $v(f^{-1})>v(x)$. Thus, applying
the case we already handled, we see that $f^{-1}$, and hence also $f$, is eventually affine.

We are left with the case $v(f)=v(x)$. Using Fact \ref{Edmundo} (3), the function $\Delta(f)$ tends to a constant.
Thus, as above, we may use the $\CM$-definable set $\Delta(X)$ to deduce that $\Delta(f)$ is eventually constant
and thus $f$ is eventually affine.

So far we handled all cases where the bounded cell in $X$ has is the graph of some function on a ray $(a,\infty)$. The same reasoning applies to
rays $(-\infty,a)$. Applying this reasoning to $X^{-1}$, we  obtain in addition those functions which are eventually constant in $X^{-1}$,
namely sets of the form $\{d\}\times R$ whose intersection with $X$ is co-unbounded in $\{d\}\times R$. The set of such $d$'s is clearly definable over $A$.

To summarize, we showed that every unbounded cell in $X$ is either contained in the graph of an eventually affine function $f$ definable in $\CM$, or in $\{d\}\times R$ for
some $d$.
By Proposition \ref{0-definable}, the function $f$ has the form $\lambda(x)+d$, for $\lambda\in \Lambda_\CM$. Thus, we have
$\lambda_1,\ldots,\lambda_k\in \Lambda_\CM$, and for each such $i=1,\ldots,k$,
the set $D_i$ of $d\in R$ such that $\Gamma(\lambda_i+d)\cap X$ is unbounded, is $\CM$-definable over $A$, and must be finite. For every such $d$,
$\Gamma(\lambda_i+d)\setminus X$ is bounded.

The above proof handles all unbounded cells, so the set
$$X\setminus (\bigcup_{i=1}^r\bigcup_{d\in D_i} \Gamma(\lambda_i+d)\, \cup\, \bigcup_{d\in D} \{d\}\times R)$$ is bounded.
\end{proof}

\subsection{The algebraic closure and definable closure in strongly bounded structures}

Even though the full ordering on $R$ is not definable, we can still prove:
\begin{thm}\label{acl=dcl} The algebraic closure in $\CM$ equals the definable closure. Moreover, if $a\in acl_{\CM}(\bar b)$
then $b$ is in the $\CL_{bd}$-definable closure of $\bar b$.\end{thm}
\proof We use $acl$, $dcl$ and $acl_{bd}$, $dcl_{bd}$, to denote the corresponding operations in $\CM$ and $\CM_{bd}$, respectively.
We shall prove by induction on $n$: If $a\in acl(b_1,\ldots, b_n)$, for some $a,b_i\in R$, then $a\in dcl_{bd}(b_1,\ldots b_n)$.

We first handle the case $n=0$, namely $a\in acl(\0)$. In this case, there is a finite $\0$-definable set $A\sub M$ such that $a\in A$.
Viewing the set $A$ in $\CR_{omin}$, we can order the elements,  $a_1<\cdots<a_n$. The interval
$(a_1,a_n)$ is a $\0$-interval, and $<\rest(a_1,a_n)$ is $\CM_{bd}$-definable over $\0$, hence each $a_i\in dcl_{bd}(\0)$.

We proceed by induction, and assume that we proved the result for $n-1$. Assume now that $a\in acl(b_1,\ldots, b_{n-1},b_{n})$. Let $X\sub R^{n+1}$
be a $\0$-definable set such that $\la  b_1,\ldots, b_n,a\ra$ and $X_{b_1,\ldots, b_n}$ has size  $m$. Without loss of generality, for every $b_n'$,
the set $X_{b_1,\ldots, b_{n-1},b_n'}$ has size $m$.

Let $b'=(b_1,\ldots, b_{n_-1})$ and consider the set  $X_{b'}=\{\la x,y\ra\in R^2:\la b',x,y\ra\in X\}$.
By our assumption, $\dim(X_{b'})\leq 1$, and $\la b_n,a\ra\in X_{b'}$.

We now apply Theorem
\ref{thm: subset of R^2}. We obtain finitely many $\0$-definable endomorphisms
$\lambda_1,\ldots, \lambda_k\in \Lambda_\CM$ and for each $i=1,\ldots,k$, we have a $b'$-definable finite set $A_i$, such that
$$X^{bd}_{b'}=X_{b'}\setminus (\bigcup_{i=1}^k\bigcup_{d\in D_i} \Gamma(\lambda_i+d))$$ is bounded in $R^2$.

Since $|b'|=n-1$, it follows by induction that every $d\in A_i$ is in $dcl_{bd}(b')$. Assume first that $\la b_n,a\ra$ is in the graph of one
of the $\lambda_i+d$, $d\in D_i$, namely $a=\lambda_i(b_n)+d$. Because $\lambda_i$ is $\0$-definable and $d\in dcl_{bd}(b')$ it follows that
$a\in dcl_{bd}(b_1,\ldots, b_n)$.

We are left with the case $\la b_n,a\ra \in X^{bd}_{b'}$. The set $X^{bd}_{b'}$ is $b'$-definable so we may assume that $X_{b'}=X^{bd}_{b'}$
is bounded (but possibly not $\0$-bounded). Let $\pi_1$, $\pi_2$, be the projection of $X_{b'}$ onto the first and second coordiantes. Each of these is a finite union of points
and pairwise disjoint bounded open intervals. Let $$\pi_1(X_{b'})=F_1\cup \bigcup_{i=1}^k (a_i,b_i), \mbox{ for $F_1$ finite and $a_1<b_1\cdots<a_k<b_k.$}.$$
and $$\pi_2(X_{b'})=F_2\cup \bigcup_{j=1}^r (c_j,d_j), \mbox{ for $F_2$ finite and $c_1<d_1\cdots<c_r<d_r.$}$$

By Theorem \ref{thm uniform bound on the length}, there is a fixed $K\in dcl(\0)$ such that for all $i,=1,\ldots, k$ and $j=1,\ldots, r$, we have
$b_i-a_i, d_j-c_j\leq K$.

By Proposition \ref{lem:endpoints}, the sets $\{a_i\}, \{b_i\}, \{c_j\}, \{d_j\}$ are all finite and $\CM$-definable over $b'$, and
thus, by induction each of these endpoints is in $dcl_{bd}(b')$. Assume that $\la b_n,a\ra\in X\cap (a_i,b_i)\times (c_j,d_j)$, for some  $i=1,\ldots, k$
and $j=1,\ldots, r$. We replace $X$ by the $b'$-definable set $X_1=X-\la a_i,c_j\ra\cap  (0,b_i-a_i)\times (0,d_j-c_j)\sub (0,K)^2$.
Notice that  $\la b_n-a_i,a-c_j\ra\in X'$, and the fiber in $X'$ over $b_n-a_i$ is finite. Because the ordering on $(0,K)$ is $\CM_{bd}$-definable
 over $\0$,
we have
$a-c_j\in dcl_{bd}(b',b_n-a_i)$, but since $a_i,c_j\in dcl_{bd}(b')$ we have $a\in dcl_{bd}(b',b_n)$. This ends the proof that $acl=dcl_{bd}$ in $\CM$.\end{proof}

\subsection{Definable subsets of $R^n$}

We are now ready to prove the main theorem, under the assumptions of Section \ref{setting assumptions}.
\begin{thm} If $X\sub R^n$ is  $\CM$-definable over $A\sub R$ then $X$ is definable in $\CM_{bd}$ over $A$.\end{thm}
\begin{proof} It is sufficient to prove the result in $\CN\succ \CM$, so by replacing $\CR_{omin}$ (thus also its reducts) by a sufficiently saturated extension,
we may assume that $\CM$ is $\omega$-saturated.

 We prove the result by induction on $n$. For $X\sub R$, the set $X$ is either bounded or co-bounded, so we may assume that it is bounded. Thus, it can be written as a disjoint union
$$(a_1,b_1)\cup \cdots \cup(a_n,b_n)\cup F,$$ with $a_1<b_1<\cdots<a_n<b_n$ and $F$ finite.
By Lemma \ref{lem:endpoints}, each $a_i$ and $b_i$ is in $acl_\CM(A)$, so by Theorem \ref{acl=dcl}, it belong to $dcl_{bd}(A)$. Similarly, $F\sub dcl_{bd}(A)$.
By Theorem \ref{thm uniform bound on the length}, there is $K\in dcl_{bd}(\0)$ such that all intervals $(a_i,b_i)$ are of length at most $K$.
But then each interval $(0,b_i-a_i)$ is contained in a $\0$-interval, hence definable in $\CM_{bd}$ over $A$, so also $(a_i,b_i)$ is $\CM_{bd}$-definable over $A$. It follows that $X$ is definable in $\CM_{bd}$.

We now use induction on $n$: Given $X\sub R^{n+1}$ that is $ \CM$-definable over $A$, we consider, for each $t\in R^n$,
the set $X_t=\{b\in R:\la t,b\ra\in X\}\sub R$.
By the case $n=1$, each $X_t$ is $\CM_{bd}$-definable over $At$. Thus, by compactness and saturation, we can find $\CL_{bd}$-formulas over $A$,
$\phi_1(t,x),\ldots, \phi_k(t,x)$ such that for every $t\in R^n$, one of the $\phi_i(t,x)$ defines $X_t$. Let
$$T_i=\{t\in R^n:\exists x (\la t,x\ra\in X \wedge \,\, \forall x (x\in X_t\leftrightarrow \phi_i(t,x)))\}.$$

The set $T_i$ is $\CM$-definable, over $A$, thus, by induction, it is $\CM_{bd}$-definable over $A$, by some  $\psi_i(t)$.
The formula $\phi_i(t,x)\wedge \psi_i(t)$ defines $X\cap T_i\times R$, thus $X$ is definable in $\CM_{bd}$ over $A$.\end{proof}

 \subsection{A comment on failure of Definable Choice in strongly bounded $\CM$}

 Recall that a structure $\CM$ has Definable Choice if for every definable family $\{X_t:t\in T\}$ of sets, there is a definable function
 $f:T\to \bigcup X_t$ such that $f(t)\in X_t$ and if $t_1=t_2$ then $f(t_1)=f(t_2)$. Equivalently, every definable equivalence relation has a definable set
 of representatives. This fails in strongly bounded $\CM$, because the relation $x Ey\Leftrightarrow y=-x$ on $R$ cannot have a definable set of representatives. If it did then
  it will contain either a positive  or a negative ray (without its inverse).

  We believe that Elimination of Imaginaries similarly fails.

\section{Conclusion: The proof of Theorem \ref{theorem: main}}
We are now ready to collect the results proved thus far in order to prove Theorem \ref{theorem: main}.


Recall that  now want to prove that the only reducts between $\CR_{lin}$
and $\mathcal{R}_{alg}$ are:
\begin{center}
$\mathcal{R}_{alg}=\langle R;+,\cdot,<\rangle $
\par\end{center}

\begin{center}
$\mathcal{R}_{sb}= \la R;+,<,\Lambda_R, \mathfrak B\rangle$
\par\end{center}

\begin{center}
$\mathcal{R}_{semi}=\langle R;+,<,\Lambda_R \rangle $
,\,\,\, $\CR_{bd}=\langle R;+,<^{*},\Lambda_R, \mathfrak B\rangle $
\par\end{center}

\begin{center}
$\mathcal R_{lin}^*=\langle R;+,<^{*},\Lambda_R \rangle $
\par\end{center}

\begin{center}
$\mathcal R_{lin}=\langle R;+,\Lambda_R\rangle$.
\par\end{center}

First, we note that using \cite{Edmundo}  we can generalize \cite[Theorem 1.1]{Pet1} from $\mathbb R$ to arbitrary real closed fields, and show:
\begin{fact}\label{fact:edmundo} Let $R$ be a real closed field. The only reduct between $\CR_{semi}$ and $\CR_{alg}$ is $\CR_{sb}$.\end{fact}
\begin{proof} Assume that $\CM$ is a reduct of $R_{alg}$ which properly expands $\CR_{semi}$. By \cite[Fact 1.6]{Edmundo}, either $\CM$ is a reduct of $\CR_{sb}$ or a real closed field $F=\la R; \oplus,\odot\ra$ whose universe is $R$ is definable in $\CM$. Assume the latter, and then since the field is semialgebraic then, again by \cite[Corollary 2.4]{Pet1}, every semialgebraic subset of $R$ is definable in $F$ and  hence in $\CM$. Thus, $\CM\dot{=} \CR_{alg}$.

If $\CM$ is a reduct of $\CR_{sb}$ which is not semilinear then by Theorem \ref{thm4:If--is}, every bounded $R$-semialgebraic set is definable in $\CM$, thus $\CM\dot{=} \CR_{sb}$.\end{proof}


\vspace{.2cm}
We now consider an arbitrary  reduct $\CM$ of $\CR_{alg}$. Our goal is to show that $\CM$ is one of the reducts in the above list.

First, if $\mathcal{M}$ is stable then by Claim \ref{claim:-is-unstable},
$\CR_{lin}\dot{=} \mathcal{M}.$ If $\CM$ is unstable then by
Theorem \ref{thm:<^* is definable in the vector space}, $<^{*}$
is definable in $\mathcal{M}$. So $\CR_{lin}^*\dot{\subseteq} \mathcal{M}.$
So, we may assume that $<^{*}$ is definable
in $\mathcal{M}$, thus $\CR_{lin}^*\dot{\subseteq} \CM$.

\vspace{.2cm}

\noindent \emph{Case 1:} $\mathcal{M}$ is strongly bounded
and $\mathcal{M} \dot{\subseteq} \mathcal{R}_{semi}$.

\vspace{.2cm}

We claim that $\CM\dot{=} \CR_{lin}^*$: Indeed, because $\CM$ is strongly bounded then,
by Theorem \ref{thm: main results},  $\CM\dot{=}\CM_{bd}$. Because $\CM\dot{\subseteq}\CR_{semi}$, every $\CM$-definable set is semilinear, and in particular this is true for each of the $\0$-bounded sets in $\CM_{bd}$. However, it is easy to verify that every bounded semilinear set is definable in $\CR_{lin}^*$, so the whole structure $\CM_{bd}$ is a reduct of $\CR_{lin}^*$, thus so is $\CM$ as well. The converse $\CR_{lin}^*\dot{\subseteq} \CM$ is already assumed.
\vspace{.2cm}

\noindent \emph{Case 2: } $\mathcal{M}$ is strongly bounded
and  $\mathcal M \dot{\nsubseteq}\mathcal{R}_{semi}$.
We claim that $\CM=\CR_{bd}$.

\vspace{.2cm}

As in Case 1, every $\mathcal{M}$-definable set
is definable in $\CM_{bd}$. Because
$\mathcal{M}$ is a reduct of $\mathcal{R}_{alg}$ then $\CM_{bd}$
is a reduct of $\CR_{bd}$ and so $\mathcal M\dot{\subseteq}\CR_{bd}$.
By the assumption that $\mathcal M\dot{\nsubseteq}\mathcal{R}_{semi}$,
we know that there is an $\mathcal{M}$-definable semialgebraic set
which is not semilinear so by Theorem \ref{thm4:If--is}, we get that
every bounded semialgebraic set is definable in $\mathcal{M}$, hence
$\CR_{bd}\dot{\subseteq}\mathcal{M}$.
\vspace{.1cm}

Next we assume that $\CM$ is {\bf not} strongly bounded.
\vspace{.2cm}

\noindent {\emph{Case 3:} $\mathcal{M}$ is }not strongly
bounded and $\mathcal M\dot{\subseteq}\mathcal{R}_{semi}$. By Lemma \ref{Lemma full linear order is definabe-1}, the linear
order $<$ is definable in $\mathcal{M}$, so, since $\CR_{semi}^*\dot{\subseteq} \CM$, we have $\mathcal{R}_{semi}\dot{=} \mathcal{M}$.

\vspace{.2cm}

\noindent {\emph{Case 4:}  $\mathcal{M}$ is} not strongly
bounded and $\mathcal M\dot{\nsubseteq}\mathcal{R}_{semi}$.
As in Case 3, the linear order $<$ is definable in $\mathcal{M}$,
so $\mathcal{R}_{semi}\dot{\subseteq} \mathcal{M}$ . So
we know that $\mathcal{M}$ is a reduct of $\mathcal{R}_{alg}$ which properly
expands $\mathcal{R}_{semi}$. By Fact \ref{fact:edmundo}, either $\mathcal M\dot{=}\mathcal{R}_{alg}$
or $\mathcal M\dot{=}\mathcal{R}_{bd}$.

This completes the proof that if $\mathcal{M}$ is a reduct of $\mathcal{R}_{alg}$
expanding $\CR_{lin}$, then it is
one of the reducts in the above diagram.
\vspace{.2cm}

It is left to see that all reducts in the above diagram are distinct.
Because $\CR_{lin}$ is stable and $\CR_{lin}^*$
is unstable, these two are distinct. Also, the fact that $\CR_{lin}^*$ and $\CR_{bd}$ are distinct is easy to verify (e.g., the unit circle is definable in $\CR_{bd}$ but not in $\CR_{lin}^*$).  The fact that $\CR_{bd}$ is different than $\CR_{sb}$ and $\CR_{semi}$ follows from the next lemma.

\begin{lem}
\label{lemma:one direction of the thrm} Let $R$ be a real closed field. If $\mathcal{B}^{*}$
is any collection of bounded subsets of $R^{n}$, $n\in\mathbb{N}$,
then $<$ is not definable in $\CM=\langle R;+,\Lambda_R,\mathcal{B}^{*}\rangle$.
\end{lem}
\begin{proof} We use a similar idea to \cite{Pet3}
Assume towards contradiction that $<$ is definable in $\CM$, and let $\CN=\la R:+,<,\Lambda_R, \mathfrak B^*\ra$.


Let $\psi(x,y,\bar a)$, $\bar a\in R$,
be the $\CM$-formula that defines $<$. Namely $$\mathcal{N}\models \forall x\forall y\:(\psi(x,y,\bar a)\leftrightarrow x<y).$$

Let $\tilde{\CN}=\la \tilde R;+,<,\Lambda_R, \mathfrak B^*\ra \succ \CN$ be an $|N|^+$-saturated elementary extension, whose reduct to the $\CM$-language is $\tilde \CM$. It follows that $\psi(x,y,\bar a)$ defines $<$ in $\tilde \CN$ as well.

We will show that there is an automorphism
of $\tilde \CM$ which fixes $\bar a$, thus leaving
$\psi(\tilde R\times \tilde R,\bar a)$
invariant,  and yet not respecting $<$, leading to a contradiction.

The group $\la \tilde R,+\ra $ is a vector space over $R$. We define an $R$-vector subspace of $\tilde{R}$
by
\[
\mathcal{A}=\{x\in \tilde{R}:\exists \alpha\in R,\:|x|<\lambda_{\alpha}(1)\}.
\]

So, by Zorn's Lemma, there exists an $R$-vector space $V\sub \tilde{R}$ such that $\tilde{R}=\mathcal{A}\varoplus V$, and  by the saturation assumption, $V$ is non-trivial. Now
we define the following automorphism of the $R$-vector space $\tilde{R}$:
On $\mathcal{A}$ we define $\tau_{1}(v)=v$, on $\left\langle V,+\right\rangle $
we define $\tau_{2}(v)=-v$, and we  let $\tau$: $\tilde{R}\longrightarrow\tilde{R}$
be: $$
\tau(v_{1}+v_{2})=\tau_{1}(v_{1})+\tau_{2}(v_{2})=v_{1}-v_{2}.$$

This automorphism fixes all elements in $\mathcal{A}$ and in particular
fixes all sets in  $\mathfrak{B^*}$ pointewise, but does not respect $<$ (as positive elements in $V$ are sent to negative ones).
In model theoretic language $\tau$ is an automorphism of the structure
$\tilde{\CM}$, which fixes $\overline{a}$ (since $\overline{a}\in\mathcal{A}$).
However, $\tau$ does not preserve $<$, contradiction.
\end{proof}

This ends the proof of Theorem \ref{theorem: main}.\qed

\section{Appendix: The proof of Fact \ref{prop: there exists open box}}
We now prove Fact \ref{prop: there exists open box}:
\begin{fact}
Let $R$ be a real closed field and  $X\subseteq R^{n}$
a definable set in an o-minimal expansion of $\la R;<,+,\cdot \ra$.
If $X$ is not definable in $\CR_{semi}$ then, in the structure $\CM=\la R;<^*,+,\Lambda_R,X\ra$
there exists a definable bounded set which is not definable in $\CR_{semi}$.
\end{fact}
\begin{proof} We believe that this is known so we shall be brief.
We prove the result by induction on $\dim(X)$, where the case $\dim X=0$ is trivially true.
Consider the affine part of $X$, $\mathscr{A}(X)$, which is definable in $\CM$.

Assume first $\aff{X}$ is not dense in $X$. Then there is an open box $U\sub R^n$ such that $U\cap X\neq \0$ and $U\cap \mathscr{A}(X)=\0$.
We claim that  $U\cap X$ is not semilinear. Indeed, if it were then $\aff{U\cap X}$ must be nonempty, but because $U\cap X$ is relatively open in $X$ then $\aff{U\cap X}=U\cap \aff{X}=\0$, contradiction.

Thus, $U\cap X$ above is not semilinear. and
this gives the desired box when $\aff{X}$ is  not dense in $X$.

We assume then that $\aff{X}$ is dense in $X$, and consider two cases: $\aff{X}$ is either semilinear or not.  If it were semilinear then necessarily $X\setminus \aff{X}$ is not semilinear,
and because of the density assumption, $\dim(X\setminus \aff{X})<\dim(X)$ and we can finish by induction.

 Thus, we are left with the case that $\aff{X}$ is not semilinear. For simplicity, we may assume now that $X=\aff{X}$. We recall the $\CM$-definable relation $a\sim b$ from the proof of Proposition \ref{0-definable}, defined by: $X$ has the same germ at $a$ and $b$, up to translation.

 Because $X=\aff{X}$, each $\sim$-class is open in $X$, thus there are finitely many classes, at least one of which is not semilinear. Thus, we may assume that $X=\aff{X}$ consists of a single $\sim$-class. It follows that there is some $R$-subspace $L\sub R^n$, $\dim L=\dim X$,  such that $X$ is contained in a finite union of cosets of $L$. Thus each definably connected component of $X$ is contained in a single such coset of $L$.

  Each $L$ is definable in $\CM$ using $\Lambda_R$, thus the intersection of $X$ with each of these cosets is definable in $\CM$. One of these intersections is not semilinear so we may assume that $X\sub c+L$, for some $c$. Because $\dim X=\dim L$, and $\aff{X}=X$, then $X$ is open in $c+L$. We claim that $Fr(X)\sub c+L$ is not semilinear: Indeed, $Fr(X)$ is a closed subset of $c+L$, and $X$ consists of finitely many components of $c+L\setminus Fr(X)$. If $Fr(X)$ were semilinear then each of its components will also be, so $X$ would be semilinear.

  Thus, $Fr(X)$ is not semilinear, and definable in $\CM$. By o-minimality, $\dim(Fr(X))<\dim(X)$, thus by induction we may find an $\CM$-definable bounded set which
is not semilinear.\end{proof}

 In fact, a stronger result is true: If $X\sub R^n$ is definable in an o-minimal expansion of the field $R$, and not semilinear then there is some bounded open box $U\sub R^n$ such that $U\cap X$ is not semilinear (we omit the proof here as we do not need it). Notice that this last statement fails if we replace ``not semilinear'' by ``not semialgebraic'', as Rolin's example from \cite{LR} shows: There exists a definable function $f:\mathbb R\to \mathbb R$ in an o-minimal expansion of the real field, such that the restriction of $f$ to every bounded interval is semialgebraic but $f$ itself is not semialgebraic.

\bibliographystyle{plain}
\bibliography{ref}

\end{document}